\documentclass[12pt,leqno]{article}
\usepackage{amssymb,amsmath,amsthm}
\usepackage[all]{xy}
\usepackage[alphabetic,lite]{amsrefs}
\numberwithin{equation}{section}

\usepackage[top=2cm,bottom=2cm,left=2cm,right=4cm,marginparsep=0.3cm,marginparwidth=3cm,includefoot]{geometry}
\usepackage{mathtools}
\usepackage{bbb}     
\usepackage{enumerate} 
\usepackage{mathrsfs}
\usepackage[colorlinks=true, pdfstartview=FitV, linkcolor=blue, citecolor=blue, urlcolor=blue]{hyperref}

\usepackage{graphicx}
\usepackage{tikz}

\begin{document}
\title{Sheaves for spacetimes: a survey\footnote{
This text is an expanded version of a talk given at IHES in honour of Yan Soibelman (June 2026) as well as at ICMAM Latin America, July 2026}}
\author{Pierre Schapira}
\maketitle

\begin{abstract}
On a spacetime endowed with a time function (such as a globally hyperbolic manifold) we  solve the global Cauchy problem for 
hyperbolic $\shd$-modules in the framework of Sato's hyperfunctions,
illustrating  the effectiveness of microlocal-sheaf theoretic methods in linear  problems. 
\end{abstract}

\tableofcontents

\bigskip
Sections 2--4 are joint work with Masaki Kashiwara (see in particular~\cite{KS90}), sections 5--6 with Beno{\^i}t Jubin (see~\cite{JS16} and also~\cite{Sc26} for complements) and section 7 is extracted  from~\cite{GKS12}.

\section{Introduction and history}
  A \emph{spacetime} is a connected manifold $M$ endowed with a
  non-degenerate quadratic form of signature $(+,-,\dots,-)$.

  \medskip
  This form defines a \emph{cone} in the tangent space $TM$; the affine
  model is the Minkowski space.

  \bigskip
  \textbf{Two classical questions:}
 (i) To define the \emph{future} of $x\in M$ --- the light cone issued
      from $x$. The classical approach uses piecewise smooth paths whose
      derivative belongs to the cone. \emph{We propose here a different
      candidate.}\\
(ii) To solve the \emph{Cauchy problem} (in spaces of
      distributions or analogues) for the wave equations associated with the
      form. \emph{We treat here hyperbolic $\shd$-modules in the space of
      Sato's hyperfunctions.}

\bigskip
There is a huge  literature concerned with 
  globally hyperbolic spacetimes and the Cauchy problem in the framework
  of distributions. Among many references, let us mention,~\cites{BGP07, BS05, BF00, FS11, Ge70, HE73, MS08}:
  
    \medskip
  \begin{center}
  \begin{minipage}{0.85\textwidth}\centering\small
    Brunetti, \ Fredenhagen, \ B\"ar, \ Ginoux, \ Pf\"affle,\\[0.4ex]
    Bernal, \ Fathi, \ Siconolfi, \ Geroch,\\[0.4ex]
    Hawking, \ Minguzzi, \ S\'anchez, \dots
  \end{minipage}
  \end{center}

  \bigskip
     The purpose of this article is to show that, in the framework of \emph{Sato's hyperfunctions}~\cite{Sa60} and
    \emph{microlocal sheaf theory}~\cite{KS90}, the global Cauchy problem for hyperbolic
    $\shd$-modules is well-posed and the proof is  \emph{particularly simple}.

  The classical Cauchy problem for hyperbolic equations is associated  with the names
  of \textbf{Hadamard}, \textbf{Leray}~\cite{Le53}, \textbf{G\aa rding}, 
  \textbf{H\"ormander}  and many others.
  \begin{itemize}
    \item \textbf{Hadamard:} the operator $\partial_t^2-\partial_x$ on $\R^2$ ---
     The  Cauchy problem with data on $\{t=0\}$ is \emph{not} well-posed in
     spaces of  $\Cd^\infty$-functions  or distributions.
    \item \textbf{H\"ormander:} has characterised the hyperbolic differential
      operators with \emph{constant} coefficients.
    \item With \emph{variable} coefficients, very little is known --- except
      for operators ``with simple characteristics'' (microlocally equivalent
      to $\partial_t$).
  \end{itemize}

  \medskip
    The situation changes dramatically when one works with Sato's 
    hyperfunctions~\cites{Sa60, SKK73} instead of distributions:
    \begin{itemize}\small
      \item Bony--Schapira~\cite{BS73} solve the local Cauchy problem for weakly hyperbolic operators.
      \item Kashiwara--Schapira~\cite{KS79}: treat the case of general systems ($\shd$-modules) and give later an essentially topological proof (as we shall see here) in~\cite{KS90}.
      \item Jubin--Schapira~\cite{JS16}  solve  the \emph{global} Cauchy problem for hyperbolic $\shd$-modules by using a time function which allows one to reduce to the local Cauchy problem (as we shall see here).
 \end{itemize}
 
 \subsection*{Acknowledgements}
 In a first version, in Examples~\ref{exa:1} and~\ref {exa:2}  the hypothesis that $N$ is compact was missing, and the result was false. This fact was pointed out by Prof. Akiyoshi Sannai from Kyoto University, who gave a counter-example with the help of  AI Math Referee.

 \section{ Microlocal sheaf theory and the micro-support}
\subsubsection*{Some notations}											
Here, $\cor$ is a unital commutative ring with finite homological dimension. For a locally compact topological space $M$,  $\Derb(\cor_M)$ denotes the bounded derived category of sheaves of $\cor$-modules on $M$. We mainly follow the notations of~\cite{KS90} and  freely use the 6 Grothendieck operations:

\spa
internal: $\cdot\tens\cdot$ and $\rhom(\cdot,\cdot)$ with adjunction $(\tens,\rhom$),

\spa
for a morphism $f\cl M\to N$: inverse images $\opb{f},\,\epb{f}$ and direct images $\roim{f},\,\reim{f}$, with adjunctions
 $(\opb{f},\roim{f})$, $(\reim{f},\epb{f})$.

\spa
For $A\subset M$ locally closed, $\cor_A$ denotes the constant sheaf on $A$ with stalk $\cor$ extended by 
  zero on $M\setminus A$.

 \spa
Here, $\omega_M$ denotes the dualizing complex on $M$, $\omega_M\eqdot\epb{a_M}\cor_\rmpt$, where $a_M\cl M\to\rmpt$. 
 If $M$ is a $\Cd^0$-manifold, $\omega_M$ is isomorphic to $\ori_M\,[\dim M]$, $\ori_M$ being the orientation sheaf. Moreover, 
for $f\cl M\to N$ a continuous map, one sets $\omega_{M/N}\eqdot\omega_M\tens\opb{f}\omega_N^{\otimes -1}$. 

\subsubsection*{Micro-support}		
Microlocal analysis was created by \textbf{Mikio Sato}~\cite{Sa70}. It was adapted to sheaf theory by  \textbf{Kashiwara--Schapira}~\cite{KS82}, who introduced the
  \emph{micro-support} of sheaves.

From now on, $M$ is a real $\Cd^\infty$-manifold, $\pi\cl T^*M\to M$ denotes  its cotangent bundle,

  \medskip
\emph{General idea.}
    For $F\in\Derb(\cor_M)$, the micro-support $\SSi(F)$ is a closed subset of
    $T^*M$ recording the codirections in which sections of $F$ \emph{do not propagate}.
    
   The micro-support has connections, and similar functorial properties,  with the wave-front set of distributions and with the characteristic variety of $\shd$-modules.  However, the wave front set lives on $\sqrt{-1}T^*M$, not on $T^*M$, and  the characteristic variety of $\shd$-modules lives on $T^*X$ for $X$ a complex manifold. See~\S~\ref{section:D} for a precise statement. 

  \begin{definition}
    Let $F\in\Derb(\cor_M)$. The micro-support $\SSi(F)\subset T^*M$ is the
    closed subset defined as follows. For $W\subset T^*M$ open,
    $W\cap\SSi(F)=\emptyset$ if and only if, for any $x_0$ and any real $\Cd^1$-function
    $\phi$ with $p=(x_0;d\phi(x_0))\in W$,
    \[
      \bigl(\rsect_{\{x\,;\,\phi(x)\geq\phi(x_0)\}}F\bigr)_{x_0}\simeq 0.
    \]
  \end{definition}

  \medskip
  In other words, $p\notin\SSi(F)$ if $F$ has no local cohomology supported by
  the closed ``half-spaces'' whose conormal lies in a neighbourhood of $p$. If $U$ is the complementary of this ``half-spaces'' and $u$ is a section of $F$ which vanishes on $U$, then $u$ vanishes on $U\cup V$ for an open neighborhood $V$ of $x_0$. 

\bigskip
\begin{center}
\begin{tikzpicture}[scale=1]
\draw[left color=blue!20, right color=blue!20,blue!30] (-2,0) arc[start angle=-60,end angle=60,x radius=4,y radius=2] coordinate[pos=0.7] (P);
\draw[fill=black] (P) circle (.03);
\draw[->,thick] (P) -- ++(40:0.5);
\node[left] at (P) {${}_{(x_0,\xi_0)}$};
\node at (-1.5,1) {$U$};
\begin{scope}[shift={(4,0)}]
\draw[left color=blue!20, right color=blue!20,blue!20] (-2,0) arc[start angle=-60,end angle=60,x radius=4,y radius=2] coordinate[pos=0.7] (P)  coordinate[pos=0.5] (A) coordinate[pos=0.9] (B);
\node at (-1.2,1) {$U\cup V$};
\draw[left color=blue!20, right color=blue!20,blue!30] 
(B) to[out=-20, in=90, looseness=2] (A);
\draw[blue!30] (-2,0) arc[start angle=-60,end angle=60,x radius=4,y radius=2];
\draw[fill=blue!30, blue!50] (P) circle (.03);
\end{scope}
\end{tikzpicture}
\end{center}
\vspace{1em}

\[
  (x_0,\xi_0)\notin\SSi(F)
  \;\Longrightarrow\;
  R\Gamma(U \cup V;\,F) \xrightarrow{\;\sim\;} R\Gamma(U;\,F)
\]


\subsubsection*{Properties of the micro-support}
  \begin{itemize}
    \item $\SSi(F)$ is $\R^+$-conic (invariant under the $\R^+$-action on
      $T^*M$).
    \item $\SSi(F)\cap T^*_MM=\pi(\SSi(F))=\supp(F)$.
    \item \textbf{Triangular inequality:} for a distinguished triangle
      $F_1\to F_2\to F_3\to[+1]$,
      \[
        \SSi(F_i)\subset\SSi(F_j)\cup\SSi(F_k)\qquad (j\neq k).
      \]
    \item $\SSi(F)$ is \beb co-isotropic \eb (also called ``involutive''), see~\cite{KS90}*{Def.~6.4.1}.
  \end{itemize}

  \bigskip
  Co-isotropy is defined via the Whitney normal cone of $\SSi(F)$ (see below); 
   \begin{definition}[co-isotropic]
   a subset  $S\subset T^*M$ is co-isotropic at $p\in T^*M$ if
    $C_p(S,S)^\perp\subset C_p(S)$, where $C_p(S,S)^\perp\subset T^*_pT^*M$ is identified with
    a subset of $T_pT^*M$ via the Hamiltonian isomorphism.
   \end{definition}
  If $\SSi(F)$ is smooth  in a neighborhood  of $p$, one recovers the usual notion: $T_p\SSi(F)$
  contains its symplectic orthogonal. In particular, $\dim T_p\SSi(F)\geq\dim M$.

\subsubsection*{Examples}
  \begin{itemize}
    \item[(i)] Let $F$ be a non-zero local system on $M$and assume that $M$ is connected. Then 
      $\SSi(F)=T^*_MM$, the zero section of $T^*M$.
    \item[(ii)] Let $N\subset M$ be a  closed submanifold. Then 
      $\SSi(\cor_N)=T^*_NM$, the conormal bundle of $N$ in $M$. Note that one can now define the conormal bundle of any locally closed subset $A\subset M$ by setting $T^*_AM\eqdot\SSi(\cor_A)$. 
    \item[(iii)] Let $X$ be a complex manifod and let $\shm$ be a coherent $\shd_X$-module. 
 Set     $F=\rhom[\shd_X](\shm,\sho_X)$. Then $\SSi(F)=\chv(\shm)$, 
the characteristic variety --- see \S~\ref{section:D}).
    \item[(iv)] $M=\R$ (hence, $T^*M=\R^2$), $F=\cor_I$ with $I$ an interval. The case $I=[0,1)$ and $I=[0,1]$ are shown on the picture.
  \end{itemize}

\begin{center}
\begin{tikzpicture}[scale=1,baseline=0cm]
\draw[dotted,thick, fill=black] (-1,0) ->(0,0);  
\draw[dotted,thick, fill=black] (1.03,0) ->(2,0);  
\draw[blue,fill=black, thick] (0,0) -> (.97,0);
\draw[fill=blue] (0,0) circle (.03);
\draw[blue] (1,0) circle (.03);
\draw[blue, thick] (0,0.03) -> (0,1.2);
\draw[blue, thick] (1,0.03) -> (1,1.2);
\draw (-.8,1.2) node {$k^{}_{[0,1[}$};
\draw (1.5,.3) node {$T^*M$};
\draw (0,-0.3) node {${}_0$};
\draw (1.1,.-0.3) node {${}_1$};
\draw[dotted,thick, fill=black] (3,0) ->(6,0);  
\draw[blue, thick] (4,0) -> (4.97,0);
\draw[fill=blue] (4,0) circle (.03);
\draw[fill=blue] (5,0) circle (.03);
\draw[blue, thick] (4,0.03) -> (4,1.2);
\draw[blue, thick] (5,-0.03) -> (5,-1.2);
\draw (3.2,1.2) node {$k^{}_{[0,1]}$};
\draw (5.5,.3) node {$T^*M$};
\draw (4,-0.3) node {${}_0$};
\draw (5.1,-0.3) node {${}_1$};
\end{tikzpicture}
\end{center}

\subsubsection*{Functoriality}
 Let $f\cl M\to N$ be a morphism of manifolds.  We get the maps
   \[
  \xymatrix@C=2.2em@R=1.8em{
  TM\ar[rd]_-{\tau}\ar[r]^-{f'}&M\times_NTN\ar[d]^-{\tau}\ar[r]^-{f_\tau}&TN\ar[d]^-{\tau}\\
  &M\ar[r]^-f&N,
  }\quad
  \xymatrix@C=2.2em@R=1.8em{
  T^*M\ar[rd]_-{\pi}&M\times_NT^*N\ar[d]^-{\pi}\ar[l]_-{f_d}\ar[r]^-{f_\pi}&T^*N\ar[d]^-{\pi}\\
  &M\ar[r]^-f&N.
  }
  \]
 
  \begin{definition}\label{def:noncar}
 Let $\Lambda\subset T^*N$ be a closed $\R^+$-conic subset. One says that $f$ is
  \emph{non-characteristic} for $\Lambda$ if
  \[
    \opb{f_\pi}\Lambda\cap\opb{f_d}T^*_MM\subset M\times_NT^*_NN.
  \]
  \end{definition}
  Since $\Lambda$ is closed and $\R^+$-conic, this is equivalent to:
  \eqn
  &&f_d\mbox{ is proper on }\opb{f_\pi}\Lambda.
  \eneqn

  \begin{theorem}\label{th:functor1}
 Let  $F\in\Derb(\cor_M)$, $G\in\Derb(\cor_N)$.
    \begin{itemize}
      \item[(a)] If $f$ is \emph{proper} on $\supp(F)$, then
        $\SSi(\roim{f}F)\subset f_\pi\,\opb{f_d}\,\SSi(F).$
      \item[(b)] If $f$ is \emph{non-characteristic} for $\SSi(G)$, then 
        $\SSi(\opb{f}G)\subset f_d\,\opb{f_\pi}\,\SSi(G)$  
and  one has the natural isomorphism  $\opb{f}G\tens\omega_{M/N}\isoto\epb{f}G$.
    \end{itemize}
  \end{theorem}

\subsubsection*{Micro-support along a submanifold}
  When $f$ is \emph{characteristic}, the inverse-image statement becomes
  delicate. We need to define the micro-support along a submanifold.
  
   \begin{definition}[Whitney normal cone]
    For $S\subset M$ locally closed and $N\subset M$ a submanifold, the Whitney normal cone of $S$ along $N$ is the closed cone of  the normal bundle $T_NM$ defined  in a coordinate system $(x',x'')$ with $N=\{x'=0\}$:
    \[
      (x''_0;v_0)\in C_N(S)\subset T_NM
    \]
    iff there is a sequence $(x_n,c_n)\in S\times\R^+$ with
    $x'_n\to0$, $x''_n\to x''_0$, $c_n x'_n\to v_0$.
  \end{definition}

  For $S_1,S_2\subset M$, $C(S_1,S_2)$ is the normal cone of
  $S_1\times S_2$ along the diagonal, $TM\simeq T_\Delta(M\times M)$.

  \medskip
  Let $j_M\cl M\into X$ be a closed embedding. The projection
  $\pi\cl T^*_MX\to M$ defines
    \[
   \xymatrix@C=2.2em@R=1.8em{
  T^*T^*_MX&T^*_MX\times_MT^*M\ar@{_(->}[l]_-{\pi_d}\ar[r]^-{\pi_\pi}&T^*M}.
  \]
  Restricting the embedding $\pi_d$ to the zero-section of $T^*_MX$ and using  the Hamiltonian isomorphism, one obtains the embedding
  \eq\label{eq:embed}
    T^*M\hookrightarrow T^*T^*_MX\simeq T_{T^*_MX}T^*X.
  \eneq
  In coordinates $(x,y;\xi,\eta)$ on $T^*X$ with $M=\{y=0\}$ and
  $T^*_MX=\{y=\xi=0\}$, this map is given by $(x;\xi)\mapsto(x,0;\xi,0)$.

  \begin{definition}
    For $F\in\Derb(\cor_X)$, set
    \[
      \SSi_M(F)=T^*M\cap C_{T^*_MX}(\SSi(F)),
    \]
    where $C_{T^*_MX}(\SSi(F))\subset T_{T^*_MX}T^*X$ is the Whitney normal cone 
    of $\SSi(F)$ along  $T^*_MX$ and $T^*M$ is embedded in $T_{T^*_MX}T^*X$  by~\eqref{eq:embed}.
  \end{definition}
  
 \begin{theorem}[\cite{KS90}*{Cor.~6.4.4, Prop.~6.2.4}]
    Let $X$ be a manifold, $M$ a closed submanifold, $F\in\Derb(\cor_X)$.
    Then
    \[
      \SSi(\rsect_MF)\subset\SSi_M(F).
    \]
  \end{theorem}

\section{ Links with $\shd$-modules and the characteristic variety}\label{section:D}
Let $(X,\sho_X)$ be a complex manifold and $\shd_X$ the sheaf of finite-order
  holomorphic differential operators on $X$. We refer to~\cite{Ka03} for an exposition of   $\shd$-module theory. 

   \medskip
  A coherent   $\shd_X$-module $\shm$ admits locally a finite free resolution
  \[
    0\to\shd_X^{N_n}\xrightarrow{\,\cdot P_{n-1}\,}\cdots\to
    \shd_X^{N_1}\xrightarrow{\,\cdot P_0\,}\shd_X^{N_0}\to\shm\to0. 
  \]
 The complex of its holomorphic  solutions  $\rhom[\shd_X](\shm,\sho_X)$ is then represented by the complex
  \[
    0\to\sho_X^{N_0}\xrightarrow{\,P_0\cdot\,}
    \sho_X^{N_1}\xrightarrow{\,P_1\cdot\,}\cdots
    \xrightarrow{\,P_{n-1}\cdot\,}\sho_X^{N_n}\to0.
  \]
  
\subsubsection*{The characteristic variety}

To $\shm\in\Derb_\coh(\shd_X)$ is associated its \emph{characteristic
  variety} $\chv(\shm)\subset T^*X$, a closed $\C^\times$-conic analytic
  variety.

  \medskip
  If $\shm=\shd_X/\shd_X\cdot P$ with $P$ of order $m$,
  \[
    \chv(\shm)=\{(x;\xi)\in T^*X\,;\,\sigma_m(P)(x;\xi)=0\},
  \]
  the zero-set of the \emph{principal symbol} of $P$.
\begin{theorem}[Sato--Kawai--Kashiwara~\cite{SKK73}]
 Let $\shm\in\Derb_\coh(\shd_X)$. Then    $\chv(\shm)$ is \beb co-isotropic\eb in $T^*X$.
    \end{theorem}
  A purely algebraic proof of this result has been given later by O.~Gabber~\cite{Ga81}; partial results were first obtained by Guillemin--Quillen--Sternberg~\cite{GQS70}.

\subsubsection*{Cauchy--Kowalevska--Kashiwara}
Let $\shm\in\Derb_\coh(\shd_X)$ and let $f\cl Y\to X$ be a morphism of complex manifolds.
One denotes by 
  $\shm_Y\eqdot\dopb{f}\shm$  the induced $\shd_Y$-module on $Y$ (see~\cite{Ka03}).

  \medskip
  One says  that  $Y$ is \emph{non-characteristic} for $\shm$ if $f$ is
  non-characteristic for $\chv(\shm)$:
  \[
    \opb{f_\pi}\chv(\shm)\cap\opb{f_d}T^*_YY\subset Y\times_XT^*_XX.
  \]
  \begin{theorem}[Kashiwara, master's thesis~\cite{Ka70}]
    If $Y$ is non-characteristic for $\shm$, then $\shm_Y$ is
    $\shd_Y$-coherent and there is a canonical isomorphism
    \[
      \opb{f}\rhom[\shd_X](\shm,\sho_X)\isoto
      \rhom[\shd_Y](\shm_Y,\sho_Y).
    \]
  \end{theorem}

\subsubsection*{Micro-support and  characteristic variety}
  \begin{theorem}[\cite{KS90}*{Th.~11.3.3}]
    Let  $\shm\in\Derb_\coh(\shd_X)$ and let $F=\rhom[\shd_X](\shm,\sho_X)$ be  its complex of
    holomorphic solutions. Then
    \[
      \SSi(F)=\chv(\shm).
    \]
  \end{theorem}
\begin{proof}[Sketch of proof]
(i)  ``$\subset$''. One reduces to the case where $\shm=\shd_X/\shd_X\cdot P$ for a differential operator $P$ and one  proves
that if $U\subset X$ is an open subset with smooth boundary in a neighborhood of $x_0\in\partial U$, non-characteristic for $P$, then any 
holomorphic solution $u$ of $Pu=0$ defined on $U$ extends holomorphically in a neighborhood of $x_0$. 
As remarked by Zerner~\cite{Ze71}, this follows from the \emph{the refined} Cauchy--Kowalevska theorem proved by Petrowsky and implicitly  used by Leray (see~\cite{Le57}). One can find a detailed proof in~\cite{Ho83}*{Th.~9.4.7}.
      
  \spa
      (ii)  ``$\supset$''.  The sheaf $\she^\R_X$ of microlocal operators
      of Sato--Kawai--Kashiwara~\cite{SKK73} is defined by
      \eqn
      \she^\R_X=\mu_\Delta(\sho_{X\times X}^{(0,d_X)})\,[d_X].
      \eneqn
      Here, $\mu_\Delta$ is the Sato's microlocalization functor along the diagonal. 
      The sheaf  $\she^\R_X$ is faithfully flat over the sheaf  $\she_X$ of micro-differential operators and one proves that
      \eqn
      &&\chv(\shm)=\supp(\she_X\tens[\opb{\pi}\shd_X]\opb{\pi}\shm).
      \eneqn
      Therefore 
        \eqn
      \chv(\shm)&=&\supp(\she^\R_X\tens[\opb{\pi}\shd_X]\opb{\pi}\shm)\\
      &=&\supp\mu_\Delta \rhom[\shd_X](\shm,\sho_{X\times X}).
      \eneqn
Finally, $\supp\mu_\Delta F\subset\SSi(F)\cap T^*_\Delta(X\times X)$ by~\cite{KS90}*{Cor.~5.4.10}.
\end{proof}

  \medskip
   As a corollary one gets another proof of the involutivity of $\chv(\shm)$ --
  at the opposite from Gabber's algebraic argument.

\section{Hyperfunctions and the local Cauchy problem}
\subsubsection*{Hyperfunctions}
  From now on, $M$ is a real analytic manifold ($\dim M=n$) and $X$ is a
  complexification of $M$. Define
  \[
    \sha_M\eqdot\sho_X\tens\C_M,\qquad
    \shb_M\eqdot\rhom(\RD'_X\C_M,\sho_X),
  \]
  the sheaves of real analytic functions and of \emph{Sato's
  hyperfunctions}.
  Here,  $\RD'_X\C_M=\rhom(\C_M,\C_X)\simeq\ori_M[-n]$. 
  
  \spa
It is proved by Sato that $\shb_M$ is concentrated in degree $0$. 
Then,  $\shb_M\simeq H^n_M(\sho_X)\tens\ori_M$.\\
  Roughly speaking, up to shift and orientation, $\shb_M=\rsect_M\sho_X$.

  Hyperfunctions may be represented as boundary values of holomorphic functions.  
 The sheaf $\shb_M$ is a flabby sheaf, much bigger than the sheaf of distributions.

\subsubsection*{The hyperbolic characteristic variety}
  \begin{definition}
    For $\shm$ a coherent left $\shd_X$-module, the \emph{hyperbolic
    characteristic variety} of $\shm$ along $M$ is
    \[
      \chv_M(\shm)=T^*M\cap C_{T^*_MX}(\chv(\shm)).
    \]
  \end{definition}
\begin{remark}
In case where $\shm=\shd_X/\shd_X\cdot P$ for a differential operator $P$, what we call ``hyperbolic'' corresponds to what was classically  called ``weakly hyperbolic''. 

Moreover, one sometimes says that $\theta\in T^*M$ is hyperbolic if $\theta$ does not belong to the hyperbolic characteristic variety. 
\end{remark}

  \medskip
  Combining the preceding results, we get
    \begin{theorem}
    Let $F=\rhom[\shd_X](\shm,\shb_M)$. Then
    \[
      \SSi(F)\subset\chv_M(\shm).
    \]
  \end{theorem}
  Hyperfunction solutions of a $\shd$-module \emph{propagate} in hyperbolic codirections, that is, outside the hyperbolic characteristic variety.
  
   \begin{corollary}\label{cor:locahyperCP}
    Let $N\subset M$ be a closed analytic submanifold, $Y$ a
    complexification of $N$ in $X$, $f\cl Y\into X$. Assume the
    \emph{hyperbolicity condition}
    \[
      T^*_NM\cap\chv_M(\shm)\subset T^*_MM.
    \]
    Then
    \[
      \rhom[\shd_X](\shm,\shb_M)\vert_N\isoto
      \rhom[\shd_Y](\shm_Y,\shb_N).
    \]
  \end{corollary}
  In other words, the Cauchy problem for hyperfunctions is well-posed for non-hyperbolic
  characteristic submanifolds. 

 \begin{proof}
  Set $n=\dim M$, $d=n-\dim N$. The proof reduces to a chain of isomorphisms:
  \eqn
  \rhom[\shd_X](\shm,\shb_M)\vert_N
    &\simeq&\rsect_N\rhom[\shd_X](\shm,\shb_M)\tens\ori_{N/M}[d]\\
    &\simeq&\rsect_N\rhom[\shd_X](\shm,\rsect_M\sho_X)\tens\ori_{N}[n+d]\\
    &\simeq&\rsect_N\rhom[\shd_X](\shm,\rsect_Y\sho_X)\tens\ori_{N}[n+d]\\
    &\simeq&\rsect_N\rhom[\shd_X](\shm,\sho_X)\vert_Y\tens\ori_{N}[n-d]\\
    &\simeq&\rsect_N\rhom[\shd_Y](\shm_Y,\sho_Y)\tens\ori_{N}[n-d]\\
    &\simeq&\rhom[\shd_Y](\shm_Y,\shb_N).
  \eneqn
  Here, the first isomorphism follows from the fact that $N$ is non-characteristic for  
  $\chv_M(\shm)$, the second from the definition of $\shb_M$, the third from the fact that 
  $N=M\cap Y$, the fourth from the fact that $Y$ is non-characteristic for  $\chv(\shm)$, the fifth from the Cauchy-Kowalevska-Kashiwara theorem and the las one from the definition  of 
  $\shb_N$. 
\end{proof}

\section{Causal manifolds and the $\lambda$-topology}

\subsubsection*{Causal manifolds}
  In a Lorentzian spacetime, the quadratic form gives  a \emph{light cone} in the tangent space. For our purposes we don't need the metric:
  only the cone, and we prefer to work in the cotangent space.

  \begin{definition}\label{def:causalmanif}
  \begin{itemize}
 \item[\rm(i)] 
 A \emph{causal manifold} $(M,\lambda)$ is a manifold $M$ together with $\lambda$, a closed convex proper cone of $T^*M$
 satisfying $T^*_MM\subset\lambda$.
\item[\rm(ii)] 
A morphism $f\cl(M,\lambda_M)\to(N,\lambda_N)$ is \emph{causal} if $f_d\opb{f_\pi}\lambda_N\subset\lambda_M$.
 \item[\rm(iii)]  
$(\R,+)$ denotes the causal manifold $\R$ with $\Lambda_\R=\{(t;\tau)\,;\,\tau\geq0\}$.
\end{itemize}
  \end{definition}
Observe that the cone $\lambda$ satisfies
\eq\label{eq:good}
&&\lambda\cap\lambda^a= T^*_MM, \quad \lambda+\lambda=\lambda.
\eneq
Here $\lambda^a=-\lambda$ is the opposite cone. 
  \begin{example}
    Let $\BBV$ be a finite-dimensional real vector space, $\theta$ a closed
    proper convex cone with non-empty interior, $V\subset\BBV$ open convex,
    $\lambda=V\times\theta^{\circ a}$. Then $(V,\lambda)$ is a causal
    manifold.
  \end{example}

\subsubsection*{The $\lambda$-topology}
  \begin{definition}
 \begin{itemize}
 \item[\rm(a)] 
 A locally closed set $A\subset M$ is a \emph{$\lambda$-set} if $\SSi(\cor_A)\subset\lambda$.
 \item[\rm(b)] 
  A \emph{$\lambda$-open} (resp.\ \emph{$\lambda$-closed}) set is an open (resp.\ closed) $\lambda$-set.
   \end{itemize}
  \end{definition}

  \begin{proposition}\label{pro:lambdatop}
    The family of $\lambda$-open sets defines a topology on $M$, called the $\lambda$-topology,  the same
    one given by the $\lambda$-closed sets.
  \end{proposition}
\begin{proof}
Recall that an open set $U$ is $\lambda$-open if and only if $M\setminus U$ is $\lambda$-closed.

\spa
(i) If $U_1$ and $U_2$ are $\lambda$-open, so is $U_1\cap U_2$ since $\cor_{U_1\cap U_2}\simeq\cor_{U_1}\tens\cor_{U_2}$ and 
$\SSi(\cor_{U_1}\tens\cor_{U_2})\subset\lambda$ by~\eqref{eq:good}.

\spa
(ii) If $U_1$ and $U_2$ are $\lambda$-open, so is $U_1\cup U_2$. This follows from (i) and the distinguished triangle
$\cor_{U_1\cap U_2}\to \cor_{U_1}\oplus\cor_{U_2}\to\cor_{U_1\cup U_2}\to[+1]$. 

\spa
(iii) Let $\{U_i\}_{i\in I}$ be a  family of $\lambda$-open subsets. Let us order $I$ by the relation $i\leq j$ if $U_i\subset U_j$. We may assume that $I$ is non-empty and by (ii) we may assume that $(I,\leq)$ is directed. It then follows from~\cite{KS90}*{Exe.~5.7} that, setting 
$U=\bigcup_{i\in I}U_i$, $\SSi(\cor_U)\subset \lambda$.

\spa
(iv) Since $\SSi(\cor_M)= T^*_MM\subset\lambda$ and $\SSi(\cor_\varnothing)=\varnothing$, both $M$ and $\varnothing$ are $\lambda$-open and 
the proof is complete. 
\end{proof}

\begin{definition}
One denotes by $M_\lambda$ the space $M$ endowed with the $\lambdaa$-topology defined in Proposition~\ref{pro:lambdatop} and by $\rho_\lambda\cl M\to M_\lambda$ the continuous map associated with the identity of the set $M$. 
\end{definition}

\begin{example}\label{exa:lambdaopen} 
\banum
\item
If $M=\BBV$ is a vector space and $\lambdaa=\BBV\times\gamma^{\circ}$ for a closed convex proper cone with non-empty interior $\gamma\subset\BBV$, then an open subset $U\subset\BBV$ is $\lambdaa$-open if and only if $U=U+\gamma$.
\item
If $\lambda=T^*_MM$, the only $\lambda$-open subsets are the connected components of $M$. 
\item
If $\lambda=T^*M$ \lp in which case, $\lambda$ is not proper\rp, the $\lambda$-topology is the usual one.
\eanum
\end{example}

\begin{remark}
In~\cite{Sc26}, one proves that, under suitable hypotheses, the functor\\
$\roim{\rhol}\cl\Derb_{\lambda}(\cor_M)\to\Derb(\cor_{M_\lambdaa})$ is an equivalence of categories with quasi-inverse $\opb{\rhol}$, generalizing a result of~\cite{KS90} which treated the case of Example~\ref{exa:lambdaopen}~(i).
\end{remark}

\subsubsection*{Futures and preorders}

\begin{definition}
Let $(M,\lambda)$ be a causal manifold.
\banum
\item
  For $A\subset M$, the \emph{future} of $A$, denoted $\futl{A}$, is the
  closure of $A$ for the $\lambda$-topology. 
  \item
  The \emph{causal preorder} on $(M,\lambda)$, denoted $\preceql $ is defined by 
    \[
    x\preceql y\quad\iff\quad y\in\futl{x}.
  \]
  \item
  A diamond in $M$ is a set $\futl{K}\cap\futla{L}$ for $K$ and $L$ compact in $M$.
  \eanum
\end{definition}
Note that the order on $(\R,+)$ (see Definition~\ref{def:causalmanif}~(iii)) is the usual order.
Indeed, the intervals $[a,+\infty)$ are $\lambda$-closed for $\lambda=\{(t;\tau);\tau\geq0\}$. 

  \medskip
  Classically one considers the ps-preorder: $x\preceqps y$ if there exists
  a piecewise smooth causal path $c\cl I\to M$ with $c(0)=x$, $c(1)=y$.
  Then, denoting by $J^+_{\rm ps}(A)$ the future set of $A$ for the
  ps-preorder, one can show that $\futla{A}$ is the closure of $J^+_{\rm ps}(A)$.

\begin{example}\label{exa:jub}
Let  $(x_1,x_2)$ be the coordinates on $\R^2$, $Z=\{(x_1,x_2);x_1\leq0, x_2=0\}$, $M=\R^2\setminus Z$. We identify $\R^2$ and $(\R^2)^*$ by using the Euclidian structure of $\R^2$. Set
\eqn
&&\theta=\{(x_1,x_2)\in \R^2;x_2\geq\vert x_1\vert\},\quad \lambda= M\times\theta.
\eneqn
Let $x_0=(-1,-1)$. One gets
\eqn
&&\futl{x_0}=(x_0+\theta)\cap  \{x_2<0\}.
\eneqn
Let $x_1=(1,1)$. One gets
\eqn
&&\futla{x_1}= (x_1+\theta^a)\cap M.
\eneqn
Note that $x_0\in \futla{x_1}$ but $x_1\notin \futl{x_0}$. 
\end{example}

\begin{proposition}\label{pro:new}
Let $f\cl (M,\lambda_M)\to(N,\lambda_N)$ be a morphism of causal manifolds. Assume that 
$f$ is non-characteristic with respect to $\lambda_N$, that is:
\eq\label{eq:noncar1}
&&\opb{f_\pi}(\lambda_N)\cap\opb{f_d}(T^*_MM)\subset M\times_NT^*_NN.
\eneq
Then  $f$ induces a continuous map $M_{\lambda_M}\to N_{\lambda_N}$.
\end{proposition}
\begin{proof}
It follows from the hypothesis and~\cite{KS90}*{Prop.~5.4.13} or Theorem~\ref{th:functor1}~(ii) that the inverse image  functor $\opb{f}$ induces a functor
\eq\label{eq:noncar2}
&& \opb{f}\cl\Derb_{\lambda_N}(\cor_{N})\to \Derb_{\lambda_M}(\cor_{M}).
\eneq
Now let $V$ be a $\lambda_N$-open subset of $N$, that is, $V$ is open and $\SSi(\cor_V)\subset \lambda_N$.
Then $\SSi(\opb{f}\cor_V)\subset\lambda_M$ by \eqref{eq:noncar2} and the result follows since 
$\opb{f}\cor_V\simeq\cor_{\opb{f}V}$.
\end{proof}

\medskip
Note that $f$ is non-characteristic with respect to $\lambda_N$ if and only if  $f$ is non-characteristic with respect to $\lambda^a_N$.

\begin{corollary}\label{cor:morpreord}
Let $f\cl (M,\lambda_M)\to(N,\lambda_N)$ be a morphism of causal manifolds and assume that 
$f$ is non-characteristic with respect to $\lambda_N$. Then $f$ preserves the causal preorders, that is,\
\eqn
&&    x\preceqlm y\mbox{ implies }    f(x)\preceqln f(y).
\eneqn
\end{corollary}
\begin{proof}
 If $f\cl X\to Y$ is continuous and $A\subset X$, then $f(\ol A)\subset \ol{f(A)}$.  Therefore, $y\in\futlm{x}$ implies $f(y)\in \futln{f(x)} $.
\end{proof}

\section{Time functions and the global Cauchy problem  }

\subsubsection*{Time functions}
  \begin{definition}
  \banum
  \item
A \emph{time function} \lp also called a \emph{Cauchy time function}\rp\, on $(M,\lambda)$ is a submersive causal morphism $\tim\cl(M,\lambda)\to(\R,+)$, proper on
 $\futl{K}$ and $\futla{K}$ for every compact $K\subset M$.
\item
 A \emph{G-causal manifold} $(M,\lambda,\tim)$ is a causal manifold endowed with a time function.
  \eanum
    \end{definition}
\begin{itemize}
    \item
Since $\tim$ is submersive, it follows from Corollary~\ref{cor:morpreord} that $\tim$ preserves the causal preorder: if $x\preceql y$, then 
$\tim(x)\leq \tim(y)$. 
\item
Since $\tim$ is causal, $d\tim(x)\in\lambda$, for any $x\in M$. 
\end{itemize}

\begin{proposition}
Let $(M,\lambda,\tim)$ be a G-causal manifold. Then diamonds are compact.
\end{proposition}
\begin{proof}
Let $K$ and $L$ be two compact subsets of $M$. It follows from Corollary~\ref{cor:morpreord} that $\futl{K}$ is contained in 
$\opb{\tim}([t_0,+\infty))$ for some $t_0$ and similarly, $\futla{L}$ is contained in $\opb{\tim}((-\infty,t_1])$ for some $t_1$. Then 
 $\futl{K}\cap\futla{L}$ is contained in $\opb{\tim}([t_0,t_1])$. Since $\tim$ is proper on  $\futl{K}$, the result follows.
\end{proof}

    \bigskip
 Classically,  a space time is  globally hyperbolic if diamonds are compact and there are no causal curves. However, 
the causal order used   for this definition is not the same as the one we have used, using the $\lambda$-topology. 
This point is discussed in~\cite{Sc26}, based on~\cite{JS16}.  
    
  Theorems of Geroch, Bernal--S\'anchez, Fathi--Siconolfi, etc. assert that 
  a globally hyperbolic Lorentzian spacetime admits a Cauchy time function.

\subsubsection*{A key propagation lemma}
  \begin{lemma}\label{le:abstglobCP}
    Let $(M,\lambda,\tim)$ be G-causal, $F\in\Derb(\cor_M)$, and assume
    \[
      \SSi(F)\cap\lambda\subset T^*_MM.
    \]
    Then $\SSi(\roim{\tim}F)\subset\{(t;\tau)\in T^*\R\,;\,\tau\leq0\}$.
  \end{lemma}
\begin{proof}[Sketch of proof]
 (i)  If $\tim$ is proper on $\supp(F)$, this is immediate from the
  functoriality of the micro-support.

\spa
(ii)   In general, one considers an increasing exhaustive family $\{K_n\}_n$ of compact subsets and one sets $Z_n=\futla{K_n}$. Then  $\SSi(\cor_{Z_n})\subset\lambda^a$, so that 
  $\SSi(F)\cap \SSi(\cor_{Z_n})^a\subset T^*_MM$ and thus 
  $\SSi(F_{Z_n})\subset\SSi(F)+\lambda^a$ and .$\SSi(F_{Z_n})\cap\lambda\subset T^*_MM$.
  Then the result follows for the sheaves $F_{Z_n}$. 
  
  \spa
(iii)  To conclude, apply~\cite{KS90}*{exe.~5.7}. 
\end{proof}

\subsubsection*{The global Cauchy problem --- abstract form}
  \begin{theorem}\label{th:abstglobCP}
    Let $(M,\lambda,\tim)$ be a G-causal manifold, $F\in\Derb(\cor_M)$ with
    \[
      \SSi(F)\cap(\lambda\cup\lambda^a)\subset T^*_MM.
    \]
    For $t_0\in\tim(M)$, set $N=\opb{\tim}(t_0)$. Then $N$ is
    non-characteristic for $F$ and there is a natural isomorphism
    \[
      \rsect(M;F)\isoto\rsect(N;F\vert_{N}).
    \]
  \end{theorem}
\begin{proof}[Sketch of proof]
Since $d\tim(x)\in\lambda$,  $\pm d\tim(x)\notin\SSi(F)$ and $N$ is  non-characteristic for $F$.

\spa
Lemma~\ref{le:abstglobCP}  applied  both to $\lambda$ and $\lambda^a$ gives
  $\SSi(\roim{\tim}F)\subset T^*_\R\R$, so $\roim{\tim}F$ is locally constant, hence constant,  on
  $\R$. Therefore, 
  $\rsect(M;F)\simeq\rsect(\R;\roim{\tim}F)\simeq(\roim{\tim}F)_{t_0}\simeq
  \rsect(N;F\vert_{N})$.
\end{proof}

\subsubsection*{The global Cauchy problem for hyperfunctions}

Applying Theorems~\ref{th:abstglobCP} and Corollary~\ref{cor:locahyperCP}, we get:

  \begin{corollary}\label{cor:glohyperCP}
    Let $(M,\lambda,\tim)$ be a real analytic G-causal manifold \lp meaning that both $M$ and $\tim$ are real analytic\rp, $X$ a
    complexification of $M$, $\shm$ a coherent $\shd_X$-module. Let
    $N=\opb{\tim}(t_0)$, $Y\subset X$ a complexification of $N$. Assume
    \[
      \chv_M(\shm)\cap\lambda\subset T^*_MM.
    \]
    Then
    \[
      \rsect\bigl(M;\rhom[\shd_X](\shm,\shb_M)\bigr)\isoto
      \rsect\bigl(N;\rhom[\shd_Y](\shm_Y,\shb_N)\bigr).
    \]
  \end{corollary}

  In other words, the Cauchy problem for hyperfunctions and hyperbolic $\shd$-modules is
  \emph{globally} well posed.
 
 \begin{remark}
 In~\cite{JS16}, one proves a generalization of Corollary~\ref{cor:glohyperCP}  to higher codimensional submanifold $N$. 
\end{remark}

\subsubsection*{ Examples (wave-type operators) }
\begin{example}\label{exa:1}
Let $N$ be real analytic and let $M=N\times\R$. 
 As usual, denote by $(t;\tau)$ the coordinates on $T^*\R$. Let $R$ be a differential
    operator on $M$ of order $\leq 2$ whose principal symbol of order $2$ is
    $\geq 0$ on $T^*_MX$ and independent of $\tau$. Then
    \[
      P=\partial_t^2-R
    \]
    is  hyperbolic in the codirections $(x,t;\pm dt)$. 

 Indeed, let $\sigma(P)$ be the principal symbol of $P$ and let $(x)$ be a local coordinate system on $N$, $(x,\sqrt{-1}\eta)$ the associated coordinates on $T^*_NY\simeq\sqrt{-1}T^*N$. Let 
 $(t;\theta+\sqrt{-1}\tau)$ denote the coordinate on $\C\tens T^*\R$. Then 
  \[
    \sigma(P)(x,t;\sqrt{-1}\eta,\sqrt{-1}\tau+\theta)
    =\theta^2-\tau ^2+\sigma_2(R)(x,t;\eta)+2\sqrt{-1}\,\theta\tau.
  \]
  Therefore, $\sigma(P)(x,t;\sqrt{-1}\eta,\sqrt{-1}\tau+\theta)\neq0$ for $\theta\neq0$.
    This implies
  \eqn
  &&(x,t;\theta)\notin C_{T^*_MX}(\opb{\sigma(P)}(0)).
  \eneqn
  This last point is not obvious but follows from the local Bochner tube theorem as remarked by Kashiwara \lp see~\cite{BS73}\rp.
  
  Now assume that  $N$ is compact and  connected. 
We choose the cone $\lambda\subset T^*M$ as
\eq\label{eq:conetaupos}
&&\lambda=T^*_NN\times\{(t;\tau);\tau\geq0\}.
\eneq
Note that this cone is closed and proper but with empty interior. Let $\tim\cl N\times\R\to\R$ denote the second projection.

Let $x_0=(y_0,t_0)\in N\times\R$. Then $\futl{x_0}=N\times[t_0,+\infty)$.
Since $N$ is compact, the map $\tim$ is a time function and $(M,\lambda,\tim)$ is  a G-causal manifold.

In this situation, Corollary~\ref{cor:glohyperCP} applies and we may solve the global Cauchy problem for hyperfunction solutions. 
\end{example}

\begin{example}\label{exa:2}
 In particular, if  $(g_t)_{t\in\R}$ is an analytic family of  Riemannian metrics on the compact manifold $N$  with Laplace--Beltrami operators $\Delta_t$ and $a(t)\geq0$, then
\[
P=\partial_t^2-a(t)\Delta_t+\text{lower order terms}
\]
satisfies the hypotheses of Example~\ref{exa:1}.
\end{example}

\begin{remark}
  Hyperbolicity in a codirection depends only on the top-order part of the operator.
  For all such  hyperbolic operators, the local Cauchy problem for
  hyperfunctions is well-posed.

\spa
 The analogous results \beb  \emph{fail}\eb  in general for distributions or $\Cd^\infty$-functions.
Indeed,  \textbf{Hadamard} has shown that the  local Cauchy problem for
$\partial_t^2-\partial_x$ on $\R^2$ with data on $\{t=0\}$ is
 \beb not  well-posed\eb  for $\Cd^\infty$-functions or distributions, contrarily to hyperfunctions.
\end{remark}

\subsection*{A counter-example}
If one removes the hypothesis that $N$ is compact in Example~\ref{exa:1}, then the conclusion is false. 

The counter-example below should be credited to Prof. Akiyoshi Sannai and AI Math Referee (an AI auditing system built mainly on Anthropic's Claude).

Let $N = \R$ with coordinate $x$, let $M = N \times\R$   with coordinates $(x, t)$. Set 
  \eqn
  && A = (1 + x^2)\partial_x, \\
  &&P = \partial_t^2 - A^2= \partial_t^2 - (1 + x^2)^2 \partial_x^2 - 2x(1 + x^2) \partial_x.
  \eneqn
 Fix $h \in \Cd^\infty_c(\R)$ with $h \neq 0 $ and $\supp h \subset [1.9, +\infty)$.  Then  set 
 \eqn
 &&u(x, t) = h(t + \arctan x). 
 \eneqn
 Since $A(\arctan x) = 1$, we get 
 \eqn
 &&A u = h'(t + \arctan x) = \partial_t u,\quad  A^2 u = \partial_t^2 u, \mbox{ thus }Pu = 0.
 \eneqn
On the other hand,  since $u(0, t) = h(t)$, $u$ is non $0$. 
However, it vanishes on the open set $\{t < 0.3\}$ which contains $N\times\{0\}$.

We have found  a non-zero function $u$ solution of $Pu=0$ with support in $\{(x,t);t\geq 0.3\}$ which shows that there is no unicity in the global Cauchy problem

\spa
It is natural to ask why Corollary~\ref{cor:glohyperCP} does not apply. Choosing $\lambda$ as in~\eqref{eq:conetaupos}, 
the future of $(x_0,t_0)$ is the set $N\times\{t\geq t_0\}$ and the projection $\tim$ is not proper. If one chooses $\lambda=\{\tau\geq\vert\xi\vert\}$, then $(M,\lambda,\tim)$ is  G-causal but $\lambda\cap \chv_M(P)$ is no more contained in the zero-section, that is, the co-directions of $\lambda$ are not all hyperbolic.

\section{Before the Big Bang}

This final section (see also~\cite{Sc23}) is speculative in nature and is intended only to illustrate the geometric flexibility of sheaf-theoretic methods.

\subsubsection*{A mathematical candidate for the Big Bang}
  Suppose $\tim\cl M\to\R_{>0}$ is a time function and each
  $N_t=\opb{\tim}(t)$ is a compact Riemannian manifold.

  \medskip
  {\bf Question.}
    What happens for $t<0$ if the diameter of $N_t$ goes to $0$ as $t\to0^+$?

  \medskip
  Mathematicians (Manin--Marcolli~\cite{MM14}) and physicists (Penrose~\cite{Pe20}) have proposed
  interpretations of the Big Bang.

  \medskip
  Here is a \emph{purely mathematical} candidate --- in fact, a model for
  any \beb phase transition.\eb 

  \bigskip
   No physics is assumed --- only sheaves and their micro-supports.

\subsubsection*{Extending the universe to $t<0$}
  Represent the universe as a closed ball in $\R^n$ (a manifold with boundary) whose radius grows
  linearly with $t$: the sheaf $\cor_{\{|x|\leq t\}}$, defined for $t\geq0$.

  \medskip
  \noindent
\begin{minipage}[t]{0.45\textwidth}
\vspace{-4cm}{ (i) The micro-support at the boundary is the \emph{interior  conormal},

\medskip\noindent
(ii) extending naturally to $t<0$ gives the \emph{exterior}
conormal --- the micro-support of the constant sheaf on the open cone.}
\end{minipage}\hfill
\begin{minipage}[t]{0.45\textwidth}
      \centering
      \includegraphics[width=\textwidth]{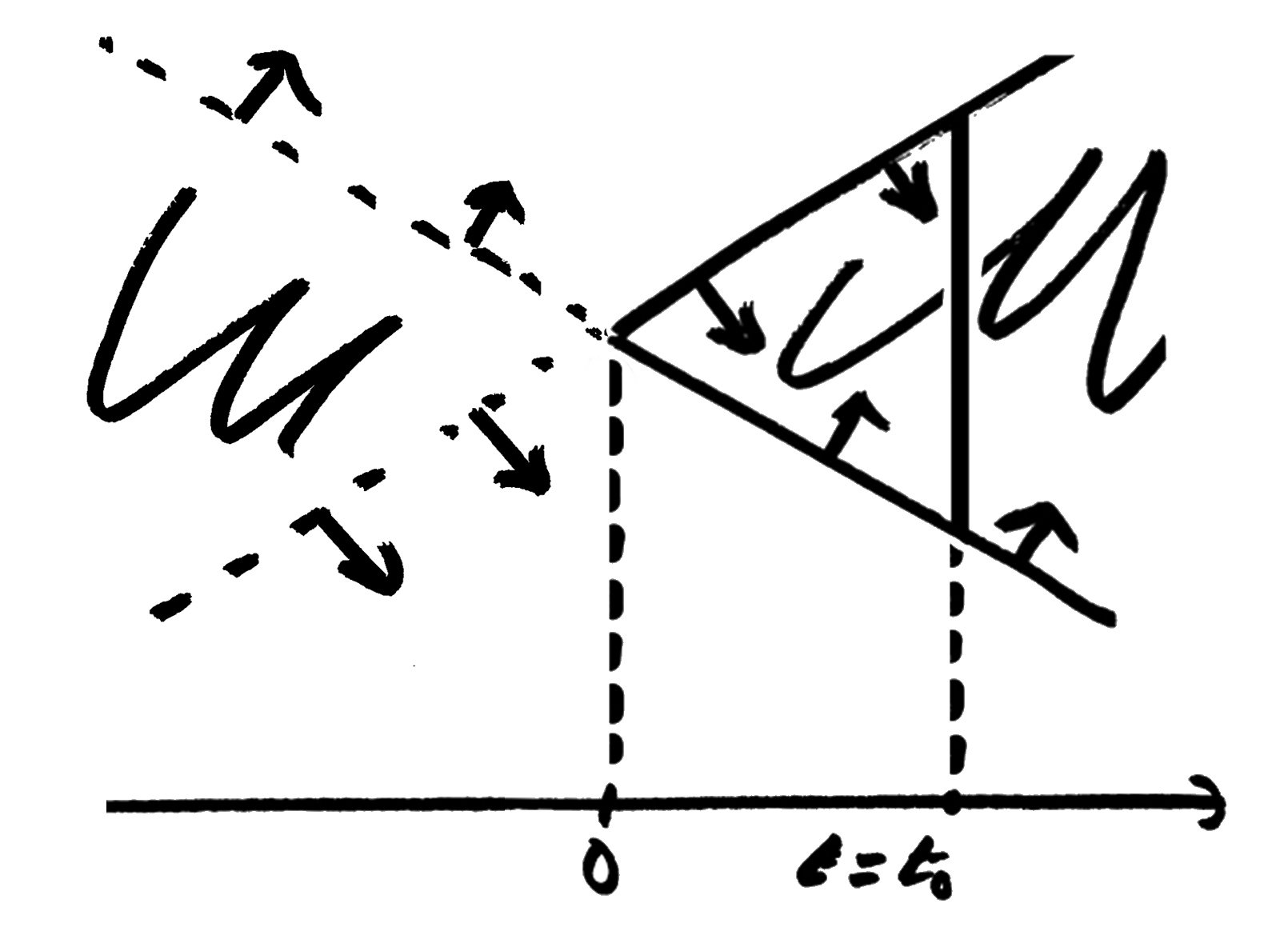}\\
      {\scriptsize  Before the Big Bang}
\end{minipage}

\subsubsection*{A distinguished triangle}
  Set $X=\R^n_x\times\R_t$. How does one glue the sheaf
  $\cor_{\{|x|\leq t\}}$ (defined for $t\geq0$) and the sheaf
  $\cor_{\{|x|<-t\}}$ (defined for $t<0$)?

  \medskip
  With S.~Guillermou and M.~Kashiwara~\cite{GKS12}*{Exa.~3.10, 3.11}, we constructed a distinguished triangle
  \[
    \cor_{\{|x|<-t\}}[n]\to K\to\cor_{\{|x|\leq t\}}\xrightarrow{\ \psi\ }[+1].
  \]
  In other words, the natural extension of the constant sheaf on the closed
  cone for $t\geq0$ is the constant sheaf on the \emph{open} cone, \emph{shifted}
  by the dimension.

\subsubsection*{The Hamiltonian isotopy}
    As we shall see below, $K$ is the \emph{quantization of a Hamiltonian isotopy}.

\medskip
  On $T^*\R^n$ with homogeneous symplectic coordinates $(x;\xi)$, consider
  \[
    \phi_t(x;\xi)=\Bigl(x-t\frac{\xi}{|\xi|}\,;\,\xi\Bigr),\qquad t\in\R.
  \]
  Outside the zero-section, $\SSi(K)$ is the smooth Lagrangian manifold
  obtained as the image of $T^*_{\{0\}}\R^n$ under this isotopy.

  \bigskip
    A phase transition (such as a Big Bang) is the instant at which the
    \emph{rank} of the projection
    $\pi\cl\Lambda\to M$ of a smooth Lagrangian
    submanifold $\Lambda\subset T^*M$ \emph{drops}.

\subsubsection*{A periodic Big Bang on the sphere}
  Replace $\R^n_x$ by a Riemannian manifold (positive convexity radius and
  injectivity radius), using the Hamiltonian isotopy associated with
  $\vvert\xi\vvert_x$.

\noindent
\begin{minipage}[t]{0.45\textwidth}
\vspace{-3cm}{ For the Euclidean $n$-sphere $M=\BBS^n$ (with $n\geq 2$): the sheaf
      obtained has a \emph{shift that jumps by the dimension at each pole}.
      Here the time is the  vertical line from the top (past) to the bottom
      (future).}
\end{minipage}\hfill%
\begin{minipage}[t]{0.45\textwidth}
   \includegraphics[width=\textwidth]{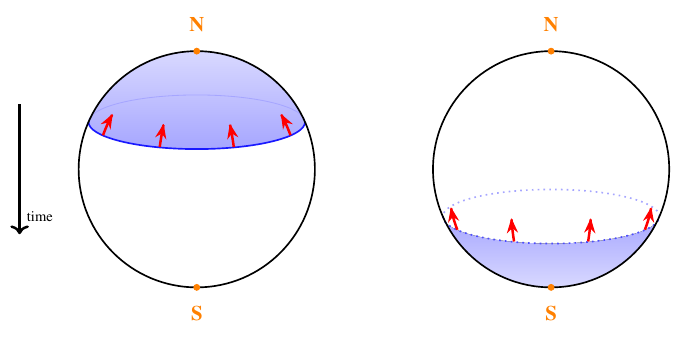}\\
      {\scriptsize Periodic Big Bang}
 \end{minipage}

\providecommand{\bysame}{\stLeavevmode\hbox to3em{\hrulefill}\thinspace}

\vspace*{1cm}
\noindent
\parbox[t]{21em}
{\scriptsize{
Pierre Schapira\\
Sorbonne Universit{\'e}, CNRS, Universit{\'e} Paris-Cit{\'e} \\
IMJ-PRG, 4 place Jussieu, 75252 Paris Cedex 05 France\\
e-mail: pierre.schapira@imj-prg.fr\\
http://webusers.imj-prg.fr/\textasciitilde pierre-schapira/
}}

\end{document}